\documentclass[11pt]{amsart}
\usepackage{geometry}                
\usepackage{graphicx,stmaryrd}
\usepackage{amssymb,amsmath}
\usepackage{epstopdf}
\usepackage{tikz}
\usepackage{lscape}
\usepackage{color}
\usepackage{todonotes}
\usepackage{mathrsfs}
\usepackage{longtable}

\theoremstyle{plain}
\newtheorem{theorem}{Theorem}[section]

\newtheorem{corollary}[theorem]{Corollary}
\newtheorem{proposition}[theorem]{Proposition}
\newtheorem{lemma}[theorem]{Lemma}

\theoremstyle{definition}
\newtheorem{definition}[theorem]{Definition}
\newtheorem{example}[theorem]{Example}
\newtheorem{remark}[theorem]{Remark}

\newtheorem{question}[theorem]{Question}

\newcommand{\B}{\mathcal{B}}
\newcommand{\M}{\mathcal{M}}
\newcommand{\N}{N}
\newcommand{\NN}{\mathbb{N}}
\newcommand{\ZZ}{\mathbb{Z}}
\renewcommand{\L}{\mathcal{L}}

\newcommand{\F}{\mathbb{F}}
\newcommand{\II}{\mathbb{I}}
\newcommand{\JJ}{\mathbb{J}}

\newcommand{\R}{\ensuremath \mathbb{R}}

\renewcommand{\S}{\ensuremath \mathcal{S}}
\newcommand{\Q}{\mathcal{Q}}
\newcommand{\sgn}{\mathrm{sgn}}

\newcommand{\J}{\mathcal{J}}

\newcommand{\dia}{\Diamond}

\newcommand{\qbin}[2]{\begin{bmatrix}{#1}\\ {#2}\end{bmatrix}_q}

\newcommand{\op}{\mathrm{op}}

\newcommand{\rk}{\mathrm{rk}}
\newcommand{\crk}{\mathrm{crk}}
\newcommand{\calP}{\mathcal{P}}

\newcommand{\bs}{\backslash}
\newcommand{\Fl}{\mathcal{F}l}

\newcommand{\ga}{\ensuremath{\alpha}}

\newcommand{\gd}{\ensuremath{\delta}}

\newcommand{\gm}{\ensuremath{\mu}}

\newcommand{\gr}{\ensuremath{\rho}}

\newcommand{\gz}{\ensuremath{\zeta}}

\title{A non-associative incidence near-ring with a generalized M\"obius function}

\author{John Johnson}
\thanks{Supported by the US Naval Academy as a Trident Scholar.}
\address{US Naval Academy\\
  572-C Holloway Rd\\
  Annapolis MD, 21402 USA}
\email[]{m213162@usna.edu}

\author{Max Wakefield}
\thanks{}
\address{US Naval Academy\\
  572-C Holloway Rd\\
  Annapolis MD, 21402 USA}
\email[]{wakefiel@usna.edu}

\begin{document}


\maketitle

{\centering\footnotesize This paper is dedicated to the memory of John Johnson.\par}

\begin{abstract} There is a convolution product on 3-variable partial flag functions of a locally finite poset that produces a generalized M\"obius function. Under the product this generalized M\"obius function is a one sided inverse of the zeta function and satisfies many generalizations of classical results. In particular we prove analogues of Phillip Hall's Theorem on the M\"obius function as an alternating sum of chain counts, Weisner's theorem, and Rota's Crosscut Theorem. A key ingredient to these results is that this function is an overlapping product of classical M\"obius functions. Using this generalized M\"obius function we define analogues of the characteristic polynomial and M\"obius polynomials for ranked lattices. We compute these polynomials for certain families of matroids and prove that this generalized M\"obius polynomial has -1 as root if the matroid is modular. Using results from Ardila and Sanchez we prove that this generalized characteristic polynomial is a matroid valuation.

\end{abstract}

\vspace{10pt}

\section{Introduction}

Combinatorial invariants in incidence algebras play a central role in many areas of combinatorics as well as in number theory, algebraic topology, algebraic geometry, and representation theory. In particular, the M\"obius function appears in the inverse of the Riemann zeta function as well as the coefficients of the chromatic polynomial for graphs. In this note we study a generalization of the classical incidence algebra by looking at three variable incidence functions. A large portion of this study is focussed on studying a 3-variable generalized M\"obius function inside this generalized incidence structure. 

Incidence algebras and M\"obius functions were popularized by Rota in \cite{rota64}. Rota characterized the classical M\"obius function from number theory (see \cite{Mobius} and \cite{Hardy-Wright}) as the inverse of the constant function {\bf 1} on the poset which is called the zeta function. In \cite{rota64} Rota gives many results on the M\"obius function including his cross-cut theorem. Since then many advances can be attributed to M\"obius functions. Of particular importance are the counting theorems of Zaslavsky in \cite{Z75} and Terao's factorization theorem (see \cite{Terao-80}) using the M\"obius function in the form of the characteristic polynomial of a hyperplane arrangement. The main motivation for this work is to build invariants which are finer than the classical M\"obius function and characteristic polynomial to obtain more information about the underlying combinatorial structure.

More recently there has been considerable developments in understanding of some classical invariants on matroids. One generalization came from Krajewski, Moffatt, and Tanasa where they built Tutte polynomials from a  Hopf algebra in \cite{KMT-18}. Taking this a little further in \cite{DFM-19} Dupont, Fink and Moci construct a categorical framework to view various combinatorial invariants where they use this framework to prove some convolution formulas. The work of Aguiar and Ardila in \cite{AA-17} framed many combinatorial structures like matroids in terms of generalized permutahedra where there is a natural Hopf monoid governing classical operations. One possible starting place for this study could be the work of Joni and Rota in \cite{JR-79}. Then in \cite{Ardila-Sanchez} Ardila and Sanchez use this Hopf monoid structure to build a concrete method for investigating valuations on many combinatorial structures. Another aim of this study is to add to add another invariant to the list of valuations and we use the methods of Ardila and Sanchez to show that one of our invariants is a valuation on matroids. One view that one can take for many combinatorial structures is that of posets (e.g. matroids are geometric lattices) and this is the view that we take here.

The starting point for our study is the collection of 3-variable functions on ordered triples of elements a poset. We equip this set of functions with the natural addition but give it a special convolution product. The motivation for this product comes from trying to symmetrize a more natural convolution product that was studied by the second author in \cite{W-flagincidence}. This product provides a kind of two sided 3-variable M\"obius function which is a sort of left inverse of the 3-variable analogue of the zeta function. We call this function the $J$-function and study many of its properties. It turns out that it is essentially a staggered product of the classical M\"obius functions and hence satisfies generalizations of many of the classical theorems on the classical M\"obius function. To prove these results we develop and use certain operations and formulas these 3-variable functions satisfy that give maps between various different types of incidence algebras. 

As a kind of application we build two different polynomials from the $J$-function: a generalized characteristic polynomial and a generalized M\"obius polynomial (see \cite{J-12} and \cite{M-12} for M\"obius polynomials). It turns out that these polynomials have some interesting properties that are not apparent from the surface. In the case of matroids the generalized characteristic polynomial has positive coefficients. Then we compute these polynomials for certain families of matroids and find special roots. Of particular interest is that the generalized M\"obius function has -1 as a root for modular matroids which mimics Theorem 1 in \cite{M-12}. However, the converse is not true and so one is led to question what do these polynomial count? Could there be some chromatic generalization for the generalized polynomials or some lattice point or finite field counting formula for these polynomials (like \cite{CF-22} or \cite{Ath-96})? 

We finish by employing the methods of Ardila and Sanchez in \cite{Ardila-Sanchez} to show that our generalized characteristic polynomial is a matroid valuation. This follows from the fact that the $J$-function splits as a product of M\"obius functions. In the case of the M\"obius polynomial we are not sure whether or not it is a valuation, yet we show that it does have a decomposition in terms of the classical characteristic polynomials. We find it interesting that this decomposition looks very similar to the recursive definition of the matroid Kazhdan-Lusztig polynomial originally defined in \cite{EPW-16}. 

We begin this study with reviewing classical results on incidence algebras and M\"obius functions in section \ref{incidence-sec}. Then we define our 3-variable incidence structure in Section \ref{near-ring-sec}. There we show that this structure has some interesting properties but that it is neither associative nor distributive. However, in Section \ref{oper-sec} we develop multiple operations which give nice formulas between these different kinds of incidence functions. Using these formulas we define a generalized M\"obius function, the $J$-function, and study its properties in Section \ref{J-func-sec}. Finally in Section \ref{chara-mob-sec} we define our generalized characteristic and M\"obius polynomials. 

{\bf Acknowledgements:} The authors are thankful for discussions with Carolyn Chun, Joel Lewis, and Will Traves. The authors are also thankful to George Andrews for help on Lemma \ref{q-john}. Frederico Ardila and Mario Sanchez significantly helped with the material on valuations for which the authors are very thankful. The authors would like to thank the US Naval Academy trident program for support during this project.

\section{Incidence Algebras}\label{incidence-sec}

Let $R$ be a commutative ring and $\calP$ be a locally finite poset. We follow \cite{Stanley-book1} and \cite{Ardila-15} for combinatorics on posets. For the remainder of this note we refer to the order in $\calP$ by $\leq$. Also, for $n\in \N$ let $[n]=\{1,2,3,\ldots ,n\}$. In this section we review basic material of incidence algebras where we follow \cite{SO97}. First we define the poset of partial flags.

\begin{definition}

The poset of \emph{partial flags of length} $k$ on $\calP$ is $$\Fl^k(\calP)=\left\{ (x_1,x_2,\ldots , x_k)\in \calP^k | \ x_1\leq x_2\leq \cdots \leq x_k\right\}$$ with order given by $(x_1,\ldots ,x_k)\preceq (y_1,\ldots ,y_k)$ if and only if for all $i\in [k]$ we have $x_i\leq y_i$.

\end{definition}

Now we define the classical incidence algebras.

\begin{definition}
The \emph{incidence algebra} on $\calP$ is the set $$\II (\calP,R)=\mathrm{Hom}(\Fl^2 (\calP),R)$$ where $R$ is a commutative ring. Addition in $\II (\calP,R)$ is given by $$(f+g)(x,y)=f(x,y)+g(x,y)$$ and the multiplication is given by convolution $$(f * g)(x,y)=\sum\limits_{x\leq a\leq b}f(x,a)g(a,y) .$$
\end{definition}

In this note we will examine multiple different operations on functions on posets. For this reason we will reserve juxtaposition only for products of elements in the ring $R$. Otherwise we will denote products of functions with specific operation names like $*$.

It turns out that $\II (\calP,R)$ is a non-commutative $R-algebra$ with identity element given by the Kronecker delta function $$\delta(x,y)=\left\{ \begin{array}{ccc}
1 & \text{  if  } & x=y\\
0 & \text{else.} &
\end{array} \right. $$
There are two other very important elements in $\II (\calP,R)$.

\begin{definition}

The \emph{zeta function} $\zeta\in \II (\calP,R)$ is defined as the constant function on $\Fl^2(\calP)$ $$\gz(x,y) =1$$ for all $(x,y)\in \Fl^2(\calP)$. The \emph{M\"obius function} $\gm \in \II (\calP,R)$ is defined by $$\sum\limits_{x\leq a\leq y} \gm (x,a) =\sum\limits_{x\leq a\leq y} \gm (a,y) = \gd(x,y)$$ for all $(x,y)\in \Fl^2(\calP)$.

\end{definition}

The M\"obius function was originally defined by M\"obius (see \cite{Mobius}) on the poset of the natural numbers ordered by division for the purpose of inverting the Riemann zeta function. Since then the M\"obius function has been used in many different contexts and broadened by the work of Rota in \cite{rota64}. For our discussions it is important to note that $\gm$ is the multiplicative inverse of the zeta function $$\gm *\gz =\gz*\gm =\gd.$$  Now we review how the incidence algebra functor factors over products. Recall that for posets $\calP$ and $\mathcal{Q}$ the product poset is $\calP\times \mathcal{Q}$ with order given by $(x_1,x_2)\leq (y_1,y_2)$ if and only if $x_1\leq y_1$ and $x_2\leq y_2$.

\begin{proposition}[Proposition 2.1.12 \cite{SO97}]\label{inc-prod}

If $\calP$ and $\mathcal{Q}$ are locally finite posets then $$\II(\calP,R) \otimes_R \II(\mathcal{Q},R)\cong \II (\calP\times \mathcal{Q} ,R).$$

\end{proposition}

Because of Proposition \ref{inc-prod} we define the following operation on functions. In order the make the exposition clear in the case when we are dealing with functions over different posets then we will put the poset in the subscript. For $f_\calP\in \II(\calP,R)$ and $g_\mathcal{Q}\in \II(\mathcal{Q},R)$ define $f_\calP\times g_\mathcal{Q}\in \II(\calP\times \mathcal{Q},R)$ by $$(f_\calP\times g_\mathcal{Q})((x_1,x_2),(y_1,y_2))=f(x_1,y_1)g(x_2,y_2).$$ Will use this notation and the following consequence of Proposition \ref{inc-prod} in our study in Section \ref{J-func-sec}. 

\begin{corollary}\label{Mob-prod}

If $\calP$ and $\mathcal{Q}$ are locally finite posets then $\gm_\calP \times \gm_\mathcal{Q}=\gm_{\calP\times\mathcal{Q}}$.

\end{corollary}

Next we recall how the M\"obius function counts chains (or is an Euler characteristic for the order complex). For $(x,y)\in \Fl^2(\calP)$ let 
$$c_i(x,y)=\left|\{(a_0,\ldots,a_i)\in \Fl^{i+1}:\forall k,\ a_k<a_{k+1} \text{ and }a_0=x \text{ and } a_i=y\}\right|$$ be the number of chains of length $i$ between $x$ and $y$. 

\begin{theorem}[Phillip Hall's Theorem \cite{Hall-OG}; Prop. 3.8.5 \cite{Stanley-book1}]\label{hall-orig}

If $\calP$ is a locally finite poset and $(x,y)\in\Fl^2(\calP)$ then $$\gm(x,y)=\sum\limits_i (-1)^ic_i(x,y).$$

\end{theorem}

Now we review Rota's Cross-cut Theorem. Let $L$ be a finite lattice with $\hat{0}$ the minimum element and $\hat{1}$ the maximum element. Usually Rota's cross-cut Theorem is stated globally in the lattice giving a formula for $\gm(\hat{0},\hat{1})$. However for our generalization we will need a local version.

\begin{definition}\label{cross-def}
Let $(x,y)\in \Fl^2(L)$. A \emph{lower cross-cut} of the interval $[x,y]=\{a\in L| x\leq a\leq y\}$ is a set $S_{x,y}\subseteq [x,y]\bs \{x\}$ such that if $b\in [x,y]\bs (\S_{x,y}\cup \{x\})$ then there is some $a\in S_{x,y}$ with $a< b$. A \emph{upper cross-cut} of the interval $[x,y]$ is a set $T_{x,y}\subseteq [x,y]\bs \{y\}$ such that if $a\in [x,y]\bs (T_{x,y}\cup \{y\})$ then there is some $b\in T_{x,y}$ with $a< b$.
\end{definition}

This definition gives Rota's famous Cross-cut theorem which we state in the style of Lemma 2.35 in \cite{OT} for use in arrangement theory. 

\begin{theorem}[Theorem 3 \cite{rota64}]\label{2-cross-thm}

If $L$ is a lattice, $(x,y)\in \Fl^2(L)$, and $S_{x,y}$ is a lower cross-cut of $[x,y]$ then $$\gm (x,y)=\sum\limits_{\substack{ A\subseteq S_{x,y}\\ \bigvee A=y}} (-1)^{|A|}.$$ Dually, if $T_{x,y}$ is an upper cross-cut of $[x,y]$ then $$\gm (x,y)=\sum\limits_{\substack{ B\subseteq T_{x,y}\\ \bigwedge B=x}} (-1)^{|B|}.$$

\end{theorem}

Next we consider Weisner's Theorem (see \cite{Weis}).

\begin{theorem}[Weisner's Theorem, Corollary 3.9.3 \cite{Stanley-book1}]\label{Weisner}

If $L$ is a finite lattice with at least two elements and $\hat{1}\neq a\in L$ then $$\sum\limits_{\substack{ x\in L \\ x\wedge a =\hat{0}}} \gm (x,\hat{1})=0.$$

\end{theorem}

Now we recall one more result that follows from these classical results for matroids: the M\"obius function alternates on matroids.

\begin{lemma}\label{alternates}

If $L$ is a finite semimodular lattice then $\sgn (\gm(x,y))=(-1)^{\rk(x)+\rk(y)}$.

\end{lemma}

\section{A 3-variable incidence non-associative near-ring}\label{near-ring-sec}

In this section we define algebraic structures for where our invariants live. It turns out that these algebraic structures support various operations that can yield nice formulas. Later these formulas will be used to show certain formulas and relations on our new invariants.

\begin{definition}

Let $R$ be a commutative ring and $\calP$ be a locally finite poset. Define the 3-variable incidence near-ring as $$\JJ(\calP ,R)=\mathrm{Hom}(\Fl^3(\calP),\R)$$ with binary operations as follows: 

\begin{itemize}

\item For $f,g\in \JJ(\calP ,R)$ we define addition by $$(f+g)(x,y,z)=f(x,y,z)+g(x,y,z).$$

\item For $f,g\in \mathcal{J}(\calP ,R)$ we define a multiplication by $$(f\Yright g)(x,y,z)=\sum\limits_{(a,b)\unlhd (x,y,z)}f(x,a,a)g(a,y,b)f(b,b,z)$$ where the juxtaposition in each term is multiplication in the ring $R$ and $(a,b)\unlhd (x,y,z)$ means $x\leq a\leq y\leq b\leq z$ in $\calP$.

\end{itemize}
\end{definition}

\begin{remark} With this $+$ the set $\JJ (\calP ,R)$ is an abelian group. It would be convenient if $\JJ (\calP ,R)$ were naturally an $R$-algebra. However, this is far from the case as we will see. Even the natural action of $R$ on $\JJ (\calP ,R)$ is flawed. Let $r\in R$ and $f,g\in \JJ(\calP,R)$ then $r\cdot (f\Yright g) = f \Yright (r\cdot g)$ but $(r\cdot f)\Yright g = r^2\cdot (f\Yright g)$.

\end{remark}

Fortunately though there are a few special functions in $\JJ (\calP, R)$ that provide substantial information. We will use these to study the structure of $\JJ(\calP ,R)$ and define other special elements later.

\begin{definition} Assume that $1$ is the multiplicative identity and 0 is the additive identity in $R$.

\begin{itemize}

\item Define $\gd_3\in \JJ (\calP, R)$ by $$ \gd_3 (x,y,z)= \left\{ \begin{array}{ccc} 1 & \ \ & \text{if } x=y=z \\
0& & \text{otherwise}.\\
\end{array}\right.$$

\item Define $\gz_3\in \JJ (\calP, R)$ by setting $\gz_3(x,y,z)=1$ for all $(x,y,z)\in \Fl^3(\calP )$.

\end{itemize}

\end{definition}

With these functions we can investigate basic properties of $\JJ(\calP ,R)$. 

\begin{proposition}\label{ident}

The element $\gd_3\in \JJ (\calP ,R)$ is a left multiplicative identity.

\end{proposition}

\begin{proof}
Let $f\in \JJ (\calP,R)$ and $(x,y,z)\in \Fl^3(\calP)$. Then \begin{align*} (\gd_3 \Yright f)(x,y,z)&=\sum\limits_{(a,b)\unlhd (x,y,z)}\gd_3(x,a,a)f(a,y,b)\gd_3(b,b,z)\\
& =\gd_3(x,x,x)f(x,y,z)\gd_3 (z,z,z)=f(x,y,z) .\end{align*}
\end{proof}

\begin{proposition}\label{noncom}

If $\calP$ is a non-trivial poset (it has at least two comparable elements) or that the base ring is not Boolean (not idempotent) then the multiplication $\Yright$ in $\JJ (\calP ,R)$ is non-commutative and $\gd_3$ is not a right multiplicative identity.

\end{proposition}

\begin{proof}

Let $(x,y,z)\in \Fl^3(\calP)$ and suppose that either $x<y$ or that $y<z$ in $\calP$ or that $R$ is not Boolean. Under these assumptions we can construct a function $f\in \JJ (\calP, R)$ that has $f(x,y,z)\neq f(x,y,y)f(y,y,z)$. Then from Proposition \ref{ident} we have $(\gd_3\Yright f)(x,y,z)=f(x,y,z)$ but $(f\Yright \gd_3)(x,y,z)=f(x,y,y)f(y,y,z)$. \end{proof}

The proof for the next fact is very similar.

\begin{proposition}\label{nonassoc}

If $\calP$ is a poset with three elements $x,y,z$ satisfying $x<y<z$ or that the base ring is not Boolean (not idempotent) then the multiplication $\Yright$ in $\JJ (\calP ,R)$ is non-associative.

\end{proposition}

\begin{proof}

Let $(x,y,z)\in \Fl^3(\calP)$ be three elements satisfying $x<y<z$ in $\calP$ or that $R$ is not Boolean. Under these assumptions we can construct a function $f\in \JJ (\calP, R)$ that has $f(x,y,y)f(y,y,y)^2f(y,y,z)\neq f(x,y,y)f(y,y,z)$. Compute \begin{align*}((f\Yright \gd_3)\Yright \gd_3))(x,y,z)&= \sum\limits_{(a,b)\unlhd (x,y,z)}(f\Yright \gd_3)(x,a,a)\gd_3(a,y,b)(f\Yright\gd_3)(b,b,z)\\
&=[(f\Yright \gd_3)(x,y,y)][(f\Yright \gd_3)(y,y,z)]\\
&= [f(x,y,y)f(y,y,y)][f(y,y,y)f(y,y,z)] .\end{align*}
 Then from Proposition \ref{ident} we have $(f\Yright (\gd_3\Yright \gd_3 ))(x,y,z)=(f\Yright \gd_3)(x,y,z)= f(x,y,y)f(y,y,z)$ which is different from $((f\Yright \gd_3)\Yright \gd_3))(x,y,z)$ by our assumption on $f$.\end{proof}
 
\begin{proposition}\label{leftdist}

If $\calP$ is any poset then the multiplication $\Yright$ in $\JJ (\calP ,R)$ is left distributive.

\end{proposition}
 
 \begin{proof}
 
 Let $f,g,h\in \JJ(\calP, R)$ and $(x,y,z)\in \Fl^3(\calP)$. Then \begin{align*} (f\Yright (g+h))(x,y,z) &= \sum\limits_{(a,b)\unlhd (x,y,z)}f(x,a,a)(g+h)(a,y,b)f(b,b,z)\\
 &=  \sum\limits_{(a,b)\unlhd (x,y,z)}f(x,a,a)(g(a,y,b)+h(a,y,b))f(b,b,z)\\
 &= \sum\limits_{(a,b)\unlhd (x,y,z)}f(x,a,a)g(a,y,b)f(b,b,z) \\
 & \  \ \ \ +\sum\limits_{(a,b)\unlhd (x,y,z)}f(x,a,a)h(a,y,b)f(b,b,z)\\
 & = (f\Yright g)(x,y,z) + (f\Yright h)(x,y,z).\end{align*} \end{proof}
 
 \begin{proposition}\label{nonrightdist}

If $\calP$ is a non-trivial poset (it has at least two comparable elements) and $R$ is any non-trivial commutative ring then the multiplication $\Yright$ in $\JJ (\calP ,R)$ is not right distributive.

\end{proposition}
 
 \begin{proof} Let $(x,y,z)\in \Fl^3(\calP)$ and $f\in \JJ (\calP ,R)$ be any function such that $f(x,y,y)+f(y,y,z)\neq 0$. Then \begin{align*}((f+\gz_3)\Yright \gd_3)(x,y,z)&=\sum\limits_{(a,b)\unlhd (x,y,z)}(f+\gz_3)(x,a,a)\gd_3(a,y,b)(f+\gz_3)(b,b,z)\\
 &=[(f+\gz_3)(x,y,y)][(f+\gz_3)(y,y,z)]\\
 & = f(x,y,y)f(y,y,z)+f(x,y,y)+f(y,y,z)+1.\end{align*} On the other hand we have \begin{align*}((f\Yright \gd_3)+(\gz_3\Yright \gd_3))(x,y,z)&=f(x,y,y)f(y,y,z)+\gz_3(x,y,y)\gz_3(y,y,z)\\
 & = f(x,y,y)f(y,y,z) +1 \end{align*} which by the hypothesis on $f$ we have the right distributive property not holding. \end{proof}
 
 With Propositions \ref{ident}, \ref{noncom}, \ref{nonassoc}, \ref{leftdist}, and \ref{nonrightdist} we conclude that $J(\calP,R)$ is a left only unital, non-commutative, non-associative, near-ring (see  \cite{PILZ-BOOK} for this terminology). Also, note that there is the zero function $Z\in\JJ (\calP,R)$ which satisfies $Z\Yright f=f\Yright Z=Z$ for all $f\in \JJ(\calP,R)$. Further note that addition in $\JJ(\calP,R)$ is by abelian. Hence $\JJ(\calP,R)$ is an abelian, zero-symmetric, left only unital, non-commutative, non-associative, near-ring. It is worth noting that in general $\JJ(\calP,R)$ is not even close to being associative on both sides and is not an alternative algebra or any similar generalization.
 
Now we look at a few special cases that do not satisfy the hypothesis of some of these propositions.
 
\begin{example}

Let $\calP=B_0=\{0\}$ be the poset with just one element and $R$ any commutative ring. Then as a set $\JJ (B_0,R) =R$, but multiplication is given by $a\Yright b=aba=a^2b$. If $R$ is Boolean then $\JJ(B_0,R)\cong R$.  Otherwise, this near-ring is not associative, not commutative, and is only left unital. 
\end{example} 

\begin{example}

Let $\calP=B_1=\{0,1\}$ be the Boolean poset of rank 1 and $R$ be any Boolean ring (one example would be $\F_2$). Then the hypothesis of Proposition \ref{nonassoc} is not satisfied and the non-equality $f(x,y,y)f(y,y,y)^2f(y,y,z)\neq f(x,y,y)f(y,y,z)$ used in the proof is always equal. It turns out that in this case $\JJ(B_1,R)$ is associative and we prove this now. In order to shorten the calculation we will denote $(0,0,0)$ by $\vec{0}$ and $(1,1,1)$ by $\vec{1}$. First we see that $$((f\Yright g)\Yright h)(\vec{0})=f(\vec{0})g(\vec{0})h(\vec{0})=(f\Yright (g\Yright h))(\vec{0}).$$

Then for the non-trivial tuple $(0,0,1)$ we compute
\begin{align*}
((f\Yright g)\Yright h)(0,0,1) = &(f\Yright g)(\vec{0})h(\vec{0})(f\Yright g)(0,0,1) \\
&+(f\Yright g)(\vec{0})h(0,0,1)(f\Yright g)(\vec{1})\\
=&  f(\vec{0})g(\vec{0})h(\vec{0})[f(\vec{0})g(\vec{0})f(0,0,1)+f(\vec{0})g(0,0,1)f(\vec{1})] \\
& +f(\vec{0})g(\vec{0})h(0,0,1)f(\vec{1})g(\vec{1})\\
=& f(\vec{0})g(\vec{0})h(\vec{0})f(0,0,1) +  f(\vec{0})g(\vec{0})h(\vec{0})g(0,0,1)f(\vec{1}) \\
&+f(\vec{0})g(\vec{0})h(0,0,1)f(\vec{1})g(\vec{1}) . 
\end{align*}

Then the other side of the associative identity is

\begin{align*}
(f\Yright (g\Yright h))(0,0,1) = &f(\vec{0})(g\Yright h)(\vec{0})f(0,0,1) +f(\vec{0})\big[(g\Yright h)(0,0,1)\big] f(\vec{1})\\
=&   f(\vec{0})g(\vec{0})h(\vec{0})f(0,0,1) + f(\vec{0})\big[g(\vec{0})h(\vec{0})g(0,0,1) \\
& + g(\vec{0})h(0,0,1)g(\vec{1})\big] f(\vec{1})\\
=&((f\Yright g)\Yright h)(0,0,1) .
\end{align*}
Hence $\JJ (B_1, R)$ is associative. This example does satisfy the hypothesis of Proposition \ref{nonrightdist}. Hence $\JJ(B_1,R)$ is a (associative) left abelian (addition is commutative) near-ring. That's about as good as it gets though. For example, if $R=\F_2$ then $\JJ (B_1,\F_2)$ is not a near-field because any function with $f(\vec{0})=0$ and $f(0,0,1)=1$ does not have an inverse. For exactly the same reason $\gd_3\in \JJ (B_1,\F_2)$ is still not a right identity element.

\end{example}

\section{Operations on incidence functions}\label{oper-sec}

In this section we look at a relationship between the classical incidence algebra $\II (\calP,R)$ and $\mathbb{J}(\calP,R)$. For $f,g \in \II(\calP, R)$ we define $f\dia g \in \JJ (\calP, R)$ by setting $$(f\dia g)(x,y,z)=f(x,y)g(y,z).$$ We can use the $\dia$ operation to construct interesting elements in $\JJ(\calP,R)$. There are relationships between the operations $*$ in $\II(\calP,R)$, $\Yright$ in $\JJ(\calP,R)$, and $\dia$.

\begin{proposition}\label{almosthom}

If $f,g,r,s\in I(\calP,R)$ and $f(b,b)g(a,a)=1$ for all $a,b\in \calP$ then $$(f\dia g)\Yright (r\dia s) =(f*r)\dia (s*g).$$ 

\end{proposition}

\begin{proof}

Let $(x,y,z)\in \Fl^3(\calP)$ and $f,g,r,s\in I(\calP,R)$. Then

\begin{align*}
((f\dia g)\Yright (r\dia s))(x,y,z)=& \sum\limits_{(a,b)\unlhd (x,y,z)}(f\dia g)(x,a,a)(r\dia s)(a,y,b)(f\dia g)(b,b,z)\\
=&\sum\limits_{(a,b)\unlhd (x,y,z)}f(x,a)g(a,a)r(a,y)s(y,b)f(b,b)g(b,z)\\
=& \left[ \sum\limits_{x\leq a\leq y} f(x,a)r(a,y)\right] \left[  \sum\limits_{y\leq b\leq z} s(y,b)g(b,z)\right] \\
=& \left[ (f*r)(x,y)\right]\left[ (s*g)(y,z)\right]\\
=& ((f*r)\dia (s*g))(x,y,z)
\end{align*} where the third equality only holds due the the assumption.\end{proof}

One can see from the proof that without the hypothesis on $f$ and $g$ that the equality will not hold. Hence there is no hope for this to give any kind of near-ring homomorphism from a twisted product version of $\II (\calP,R)\times \II (\calP,R)$. Also, the natural addition homomorphism assumption does not hold. Instead we have the following proposition which does not have special hypothesis on the functions. For this proposition there are two different additions, for $\II(\calP,R)$ and $\JJ (\calP,R)$, which for brevity we use the same addition symbol.

\begin{proposition}\label{addhom}

If $f,g,r,s\in I(\calP,R)$ then $$(f + g)\dia (r + s) =(f\dia r)+ (f\dia s) + (g \dia r) + (g \dia s).$$ 

\end{proposition}

\begin{proof}

For all $(x,y,z)\in \Fl^3(\calP)$ 

\begin{align*}
((f + g)\dia (r + s))(x,y,z)=& (f(x,y)+g(x,y))(r(y,z)+s(y,z))\\
=&f(x,y)r(y,z)+f(x,y)s(y,z)+g(x,y)r(y,z)+g(x,y)s(y,z)\\
=& ((f\dia r)+(f\dia s)+(g\dia r)+(g\dia s))(x,y,z) 
\end{align*} which is the identity we are looking for.\end{proof}

Next we show how the $\dia$ operation works over products of posets.

\begin{proposition}\label{dia-prod}

If $\calP$ and $\mathcal{Q}$ are locally finite posets, $f_\calP,g_\calP\in \II(\calP ,R)$, and $r_\Q,s_\Q\in \II(\mathcal{Q},R)$ then $$(f_\calP\dia g_\calP)\times (r_\Q\dia s_\Q)=(f_\calP\times r_\Q) \dia (g_\calP\times s_Q).$$

\end{proposition}

\begin{proof}
Let $((x_1,x_2),(y_1,y_2),(z_1,z_2))\in \Fl^3(\calP\times \mathcal{Q})$. Then $$((f\dia g)\times (r\dia s)) ((x_1,x_2),(y_1,y_2),(z_1,z_2))$$ 
$$= \left[(f\dia g)(x_1,y_1,z_1)\right] \left[ (r\dia s)(x_2,y_2,z_2)\right]$$
$$=\left[ f(x_1,y_1)g(y_1,z_1)\right] \left[ r(x_2,y_2)s(y_2,z_2)\right]$$
$$= \left[f(x_1,y_1)r(x_2,y_2)\right]\left[g(y_1,z_1)s(y_2,z_2)\right]$$
$$=\left[(f\times r)((x_1,y_1),(x_2,y_2))\right]\left[(g\times s)((y_1,z_1),(y_2,z_2))\right]$$
$$= ((f\times r)\dia (g\times s))((x_1,x_2),(y_1,y_2),(z_1,z_2))$$
which completes the proof.\end{proof}

We can also define products of functions on products of posets over 3-flags. We prefer to limit our study of $\JJ (\calP,R)$ to this product definition since the technicalities of tensor products over non-associative near-rings would present significant and unnecessary technicalities. 

\begin{definition}

Let $\calP$ and $\Q$ be locally finite posets, $f_\calP\in \JJ (\calP,R)$, and $g_\Q\in \JJ (\Q,R)$. Define $f_\calP\times g_\Q\in \JJ (\calP\times \Q,R)$ by $$(f_\calP\times g_\Q)((x_1,x_2),(y_1,y_2),(z_1,z_2))=f_\calP (x_1,y_1,z_1)g_\Q(x_2,y_2,z_2).$$

\end{definition}

Similarly to Proposition \ref{dia-prod} we get a factorization of $\times$ through $\Yright$. We use subscripts on the operations to differentiate which ring the operation occurs.

\begin{proposition}

If $\calP$ and $\Q$ be locally finite posets, $f_\calP, g_\calP\in \JJ (\calP,R)$, and $r_\Q,s_\Q\in \JJ (\Q,R)$ then $$(f_\calP \Yright_\calP g_\calP)\times (r_\Q\Yright_\Q s_\Q)=(f_\calP \times r_\Q)\Yright_{\calP\times \Q} (g_\calP \times s_\Q).$$
\end{proposition}

\begin{proof}

Let $\overline{x}=(x_1,x_2),\overline{y}=(y_1,y_2), \overline{z}=(z_1,z_2)\in \Fl^3(\calP\times \Q)$ and $\overline{a}=(a_1,a_2),\overline{b}=(b_1,b_2)\in \Fl^2(\calP\times \Q)$. Then

\begin{multline*}((f_\calP \times r_\Q)\Yright_{\calP\times \Q} (g_\calP \times s_\Q)) (\overline{x},\overline{y},\overline{z})\\
=\sum\limits_{(\overline{a},\overline{b})\unlhd (\overline{x},\overline{y},\overline{z})}(f_\calP \times r_\Q)(\overline{x},\overline{a},\overline{a}) (g_\calP \times s_\Q)(\overline{a},\overline{y},\overline{b})(f_\calP \times r_\Q)(\overline{b},\overline{b},\overline{z})\\
=\sum\limits_{(a_1,b_1)}\sum\limits_{(a_2,b_2)}f_\calP(x_1,a_1,a_1)g_\calP(a_1,y_1,b_1)f_\calP(b_1,b_1,z_1)r_\Q(x_2,a_2,a_2)s_\Q(a_2,y_2,b_2)r_\Q(b_2,b_2,z_2)\\
=\left[(f_\calP \Yright g_\calP)(x_1,y_1,z_1)\right]\left[ (r_\Q\Yright s_\Q)(x_2,y_2,z_2)\right]\\
=((f_\calP \Yright_\calP g_\calP)\times (r_\Q\Yright_\Q s_\Q))(\overline{x},\overline{y},\overline{z}) \\
\end{multline*} which is the required identity. \end{proof}

%
%

\section{The J-function}\label{J-func-sec}

Let $\calP$ be a locally finite poset. In this section we define the central invariant of this note which we call the $J$ function. This function is a generalization of the classical M\"obius function $\gm$. We show that it satisfies generalizations of the classical theorems on $\gm$. A key ingredient for these results is the operation $\dia$.

\begin{definition}

Define $J: \Fl^3(\calP) \to \mathbb{Z}$ for all fixed $(x,y,z)\in \Fl^3(L)$ by $$\sum\limits_{(a,b)\unlhd (x,y,z)} J(a,y,b)=\delta_{3}(x,y,z). $$

\end{definition}

This function is well defined because either $x=y=z$ with $J(x,y,z)=1$ or otherwise one of the following summations is non-empty and all are finite
\begin{align*}
J(x,y,z) =& -\sum\limits_{x<a<y}\left[\sum\limits_{y<b<z}J(a,y,b)\right]\\
&-\sum\limits_{x<a\leq y}J(a,y,z) -\sum\limits_{y\leq b<z}J(x,y,b).
\end{align*} Note that $J$ is exactly the function in $\JJ (\calP, R)$ such that \begin{equation}\label{def-equation}
\gz_3 \Yright J =\gd_3 .\end{equation} This is a good reason why we say it is a generalization of the classical M\"obius function and below we show that there are a few more interesting reasons. It turns out that this function was actually defined before in \cite{W-flagincidence} with the notation $\gm_3^p$ and is exactly given by the $\dia$ product construction in the previous section. 

\begin{theorem}\label{decomp}

For any locally finite poset $\calP$ we have $J =\gm_3^p = \gm \dia \gm$.

\end{theorem}

\begin{proof}

This follows from Proposition \ref{almosthom} since $\gz\in \II(\calP,R)$ satisfies the hypothesis and $$\gz_3\Yright (\gm\dia \gm)=(\gz \dia \gz)\Yright (\gm\dia \gm)=(\gz *\gm)\dia (\gm*\gz)=\gd \dia \gd=\gd_3.$$ Hence $J$ and $\gm\dia \gm$ satisfy the same recursive definition. \end{proof}

Now we can use all the classical properties of $\gm$ to conclude information about $J$. We start by noticing that $J$ is also a left inverse of $\gz_3$.

\begin{corollary}\label{zeta-right}
$J\Yright \gz_3=\gd_3$.
\end{corollary}

\begin{proof} Since $\gm$ satisfies the hypothesis of Proposition \ref{almosthom} we get $$J\Yright \gz_3=(\gm \dia \gm)\Yright (\gz \dia \gz)= (\gm *\gz)\dia (\gz*\gm)=\gd\dia \gd=\gd_3.$$ \end{proof}

Interpreting Corollary \ref{zeta-right} in terms of the definition and sums in the ring $R$ we get the following.

\begin{corollary}\label{otherside}

For any locally finite poset $\calP$ and $(x,y,z)\in \Fl^3(\calP)$ we have $$\sum\limits_{(a,b)\unlhd (x,y,z)}J(x,a,a)J(b,b,z)=\delta_{3}(x,y,z)$$ and in particular $$\sum\limits_{(a,b)\unlhd (x,y,z)}\gm(x,a)\gm(b,z)=\delta_{3}(x,y,z).$$

\end{corollary}

Now we look at how the $J$ function behaves over products. It turns out that $J$ factors over products.

\begin{proposition}\label{prod-J-func}

If $\calP$ and $\Q$ are locally finite posets then $J_\calP \times J_\Q=J_{\calP \times \Q}.$

\end{proposition}

\begin{proof}

For posets $\calP$ and $\Q$ we have $J_\calP \times J_\Q=(\gm_\calP \dia \gm_\calP)\times (\gm_\Q \dia \gm_\Q)$ by definition. By  Proposition \ref{dia-prod} $(\gm_\calP \dia \gm_\calP)\times (\gm_\Q \dia \gm_\Q)=(\gm_\calP \times \gm_\Q)\dia (\gm_\calP \times \gm_\Q)$. Then using Proposition \ref{Mob-prod} we get $(\gm_\calP \times \gm_\Q)\dia (\gm_\calP \times \gm_\Q)=\gm_{\calP\times \Q}\dia \gm_{\calP\times \Q}=J_{\calP\times \Q}.$\end{proof}

Next we look at a generalization of Phillip Hall's Theorem. For $(x,y,z)\in \Fl^3(\calP)$ set 
$$c_{i,j}(x,y,z) =\left| \{ (a_0,\ldots ,a_{i+j})\in \Fl^{i+j+1}: \forall k, \ a_k<a_{k+1} \text{ and } a_0=x,a_i=y,a_{i+j}=z\}\right|.$$ There is a bijection between the underlying set of $c_{i,j}(x,y,z)$ to the product of the under lying sets of $c_i(x,y)$ and $c_j(y,z)$. This results in the following.

\begin{lemma}\label{chain-prod}
If $\calP$ is a locally finite poset and $(x,y,z)\in \Fl^3(\calP)$ then $c_{i,j}(x,y,z)=c_i(x,y)c_j(y,z).$
\end{lemma}

This leads to a generalization of Phillip Hall's Theorem for the $J$ function.

\begin{theorem}\label{hall-gen}
If $\calP$ is a locally finite poset and $(x,y,z)\in \Fl^3(\calP)$ then $$J(x,y,z)=\sum\limits_{i,j\in \NN} (-1)^{i+j}c_{i,j}(x,y,z) .$$
\end{theorem}

\begin{proof} Let $(x,y,z)\in \Fl^3(\calP)$. By Theorem \ref{decomp} $J(x,y,z)=\gm(x,y)\gm(y,z).$ Then using Theorem \ref{hall-orig} we get \begin{align*}J(x,y,z)=&\left[\sum\limits_{i\in \NN}(-1)^ic_i(x,y)\right] \left[ \sum\limits_{j\in \NN}(-1)^jc_j(y,z)\right] \\
=&\sum\limits_{i,j\in \NN}(-1)^{i+j}c_i(x,y)c_i(y,z).
\end{align*} Lemma \ref{chain-prod} finishes the proof.\end{proof}

Now we focus on a version of Rota's cross-cut theorem for the $J$ function. We state this following the style of Lemma 2.35 in \cite{OT} and Theorem 2.4.9 in \cite{M-lec-12} which are forms of Rota's original Cross-cut Theorem in \cite{rota64}. To state this result we need the following definition.

\begin{definition}\label{3-cut-def}

Let $L$ be a finite lattice, $(x,y,z)\in \Fl^3(\L)$, $S_{x,y}$ be a lower cross-cut of $[x,y]$, and $S_{y,z}$ be a lower cross-cut of $[y,z]$ as in Definition \ref{cross-def}. We call $S_{x,y,z}=S_{x,y}\bigsqcup S_{y,z}$ a double lower cross cut of $(x,y,z)$ and call $S_{x,y}$ and $S_{y,z}$ the components of $S_{x,y,z}$. Similarly we can define $T_{x,y,z}=T_{x,y}\bigsqcup T_{y,z}$ (as well as $ST_{x,y,z}=S_{x,y}\bigsqcup T_{y,z}$ and $TS_{x,y,z}=T_{x,y}\bigsqcup S_{y,z}$).

\end{definition}

\begin{theorem}\label{3-cut-thm}

If $L$ is a finite lattice, $(x,y,z)\in \Fl^3(\L)$, $S_{x,y,z}$ is a double lower cross-cut of $(x,y,z)$ with components $S_{x,y}$ and $S_{y,z}$ then $$J(x,y,z)=\sum\limits_{\substack{ A\subseteq S_{x,y,z} \\
\bigvee (A\cap S_{x,y})=y\\
\bigvee (A\cap S_{y,z})=z }} \hspace{-0.5cm}(-1)^{|A|} .$$

\end{theorem}

\begin{proof}

Again we use Theorem \ref{decomp} together with the classical Theorem \ref{2-cross-thm} \begin{align*} J(x,y,z)=&\gm (x,y)\gm(y,z)\\
=& \left[ \sum\limits_{\substack{A_1\subseteq S_{x,y}\\ \bigvee A_1=y}}(-1)^{|A_1|}\right] \left[ \sum\limits_{\substack{A_2\subseteq S_{y,z}\\ \bigvee A_2=z}}(-1)^{|A_2|}\right] \\
=& \sum\limits_{\substack{A_1\subseteq S_{x,y}\\ \bigvee A_1=y}}\sum\limits_{\substack{A_2\subseteq S_{y,z}\\ \bigvee A_2=z}} (-1)^{|A_1|+|A_2|} .
\end{align*} Since the union in Definition \ref{3-cut-def} is disjoint $|A_1|+|A_2|=|A_1\bigsqcup A_2|$ and we have finished the proof. \end{proof}

We end this section with a generalization of Weisner's Theorem \ref{Weisner}. The interesting observation of this fact is that the middle variable of the function is crucial.

%
%
%
%
%

\begin{theorem}\label{Weisner-gen}
If $L$ is a finite lattice with at least three elements and  $\hat{0}<a<b\in L$ then  $$\sum\limits_{\substack{x\in L \\ x \wedge a=\hat{0}}}J(x,b,\hat{1})=0.$$
\end{theorem}

\begin{proof}

We compute the sum again using Theorem \ref{decomp}: \begin{align*}\sum\limits_{\substack{x\in L \\ x \wedge a=\hat{0}}}J(x,b,\hat{1}) &=\sum\limits_{\substack{x \\ x \wedge a=\hat{0}}}\gm(x,b)\gm(b,\hat{1})\\
&=\gm(b,\hat{1})\sum\limits_{\substack{x \\ x \wedge a=\hat{0}}}\gm(x,b)\\
&=\gm(b,\hat{1})\cdot 0=0
\end{align*} since $a<b$ we can apply Weisner's Theorem \ref{Weisner}. \end{proof}

\begin{remark}

There is a dual version of this result for where we sum over the left most variable as in \cite{rota64}. However, we do not see a version that sums over the middle variable.

\end{remark}


\section{Generalized characteristic and M\"obius polynomials}\label{chara-mob-sec}

In this section we examine two polynomials defined by summing over all values of the $J$ function on a ranked poset. One mimics the characteristic polynomial of a matroid and the other looks like a one variable M\"obius polynomial. We find more interesting information inside the generalized M\"obius polynomial than the generalized characteristic polynomial. That is opposite of the state of affairs in the literature on the classical polynomials, but we do not know why.

\begin{definition}

For $\calP$ a ranked finite poset with minimum element $\hat{0}$ and maximum element $\hat{1}$ the $J$-characteristic polynomial of $\calP$ is $$\J(\calP,t)=(-1)^{\rk(\calP)}\sum\limits_{x\in \calP}J(\hat{0},x,\hat{1})t^{\rk(\calP)-\rk (x)}.$$

\end{definition}

\begin{definition}

Let $\calP$ be a ranked finite poset and for $(x,y,z)\in \Fl^3(\calP)$ let $\gr (x,y,z)=3\rk(\calP)-\rk (x)-\rk(y)-\rk(z)$. The $J$-M\"obius polynomial of $\calP$ is $$\M(\calP,t)=\sum\limits_{(x,y,z)\in\Fl^3(\calP)}J(x,y,z)t^{\gr (x,y,z)}.$$

\end{definition}

We may sometimes refer to $\rk(\calP)-\rk (x)$ as $\crk(x)$. These polynomials satisfy some nice basic properties. For example it turns out that the coefficients of $\J(\calP,t)$ are positive for nice $\calP$. For convenience if $L$ is a ranked poset let $L_k=\{x\in L| \rk(x)=k\}$. 

\begin{proposition}\label{positive-coeffs}

If $L$ is a finite semimodular lattice then the coefficients of $\J(L,t)$ are positive.

\end{proposition}

\begin{proof} Using Theorem \ref{decomp} we get that $$\J(L,t)=(-1)^{\rk(L)}\sum\limits_{x\in L}\gm(\hat{0},x)\gm(x,\hat{1})t^{\rk(L)-\rk (x)}.$$ So, the coefficient of $t^k$ is $$c_k = (-1)^{\rk (L)}\sum\limits_{x\in L_k}\gm(\hat{0},x)\gm(x,\hat{1}).$$ Then note that by applying Lemma \ref{alternates} we have $$\sgn( \gm(\hat{0},x)\gm(x,\hat{1}))=(-1)^{\rk(\hat{0}) +\rk(x)}(-1)^{\rk(x)+\rk(\hat{1})}=(-1)^{\rk(L)}.$$ Hence $\sgn (c_k)=(-1)^{2\rk(L)}=1$. \end{proof}

Now we look at a foundational property for the $J$-M\"obius polynomial.

\begin{proposition}\label{root1}

If $L$ is a finite lattice with at least two elements then $\M(L,1)=0$.

\end{proposition}

\begin{proof} Since $L$ is a finite lattice with at least two elements we know there is a minimum element $\hat{0}$ and a maximum element $\hat{1}$. Then 
\begin{align*}\M(\calP,1)=&\sum\limits_{(x,y,z)\in\Fl^3(L)}J(x,y,z)\\
=&\sum\limits_{y\in L}\left[\sum\limits_{(x,z)\unlhd (\hat{0},y,\hat{1})}J(x,y,z)\right]\\
=&\sum\limits_{y\in L}\left[\gd_3(\hat{0},y,\hat{1})\right] .\\
\end{align*} Since $L$ has at least two elements $\hat{0}\neq \hat{1}$ so $\gd_3(\hat{0},y,\hat{1})$ is zero for all $y$.\end{proof}

We also have products formulas for both of these polynomials.

\begin{proposition}\label{prod-J-char}

If $\calP$ and $\Q$ are ranked finite posets then $\J(\calP\times \Q,t)=\J(\calP,t)\J(\Q,t)$.

\end{proposition}

\begin{proof} Using Proposition \ref{prod-J-func} we get that 
\begin{align*}\J(\calP,t)\J(\Q,t) = & \left[(-1)^{\rk(\calP)}\sum\limits_{p\in \calP}J_\calP(\hat{0},p,\hat{1})t^{\crk(p)}\right] \left[(-1)^{\rk(\Q)}\sum\limits_{q\in \Q}J_\Q(\hat{0},q,\hat{1})t^{\crk(q)}\right] \\
=& (-1)^{\rk(\calP)+\rk (\Q)}\sum\limits_{p\in \calP}\sum\limits_{q\in \Q}J_\calP(\hat{0},p,\hat{1})J_\Q(\hat{0},q,\hat{1})t^{\crk(p)+\crk(q)}\\
=& (-1)^{\rk(\calP\times \Q)}\sum\limits_{(p,q)\in \calP\times \Q}J_{\calP\times \Q}((\hat{0},\hat{0}),(p,q),(\hat{1},\hat{1}))t^{\crk(p,q)}\\
=&\J(\calP\times \Q,t) .
\end{align*}\end{proof}

The proof of the following is almost identical.

\begin{proposition}\label{prod-J-mob}

If $\calP$ and $\Q$ are ranked finite posets then $\M(\calP\times \Q,t)=\M(\calP,t)\M(\Q,t)$.

\end{proposition}

Now we can use these product formulas to establish formulas for Boolean matroids.

\begin{proposition}\label{J-boolean}

If $B_n$ is the Boolean lattice then $$\J(B_n,t)=(t+1)^n.$$
\end{proposition}

\begin{proof}

We start with $B_1$. This poset has two elements $B_1=\{0,1\}$. So, $\J(B_1,t)=(-1)(J(0,0,1)t^1+J(0,1,1)t^0)=t+1$. Then the result follows since $B_n=(B_1)^n$.\end{proof}

\begin{proposition}\label{M-boolean}

If $B_n$ is the Boolean lattice then $$\M(B_n,t)=(t+1)^n(t-1)^{2n}.$$
\end{proposition}

\begin{proof}

Again we first compute $\M(B_1,t)$. The only coefficients are $J(0,0,0)=1$, $J(0,0,1)=-1$, $J(0,1,1)=-1$ and $J(1,1,1)=1$. Then the result follows from \begin{align*}\M(B_1,t)=&J(0,0,0)t^3+J(0,0,1)t^2+J(0,1,1)t+J(1,1,1)\\
=&t^3-t^2-t+1\\
=&(t+1)(t-1)^2\\
\end{align*} and the application of Proposition \ref{prod-J-mob}.\end{proof}

\begin{proposition}\label{rank2M}
Let $\calP_n$ be a geometric lattice of rank two with $n$ atoms (rank 2 matroid with $n$ elements a.k.a. $U_{2,n}$). Then $\M(\calP_n,t)= (t^2-nt+1)(t+1)^2(t-1)^2$.
\end{proposition}

\begin{proof}

We prove this by induction on $n$. The base case is $n=2$ and is given by the $n=2$ version of Proposition \ref{M-boolean}. Now assume $n>2$. The lattice $\calP_{n}$ consists of $\hat{0}$, $\hat{1}$, and $n$ atoms $\alpha_1,\ldots,\alpha_{n}$. Now $J_{\calP_n}(\hat{0},\hat{0},\hat{1})=n-1$ and $J_{\calP_n}(\hat{0},\hat{1},\hat{1})=n-1$ are the only $J_{\calP_n}$ values that do not have $\alpha_{n}$ as an entry and incorporate $\alpha_{n}$ in it's recursive definition. So, $J_{\calP_n}(\hat{0},\hat{0},\hat{1})=J_{\calP_{n-1}}(\hat{0},\hat{0},\hat{1})+1$ and similarly for $(\hat{0},\hat{1},\hat{1})$. Incorporating this difference into the calculation we get that \begin{align*}\M(\calP_n,t)=&\M(\calP_{n-1},t)+t^4+t^2+J(\hat{0},\hat{0},\ga_n)t^5+J(\hat{0},\ga_n,\ga_n)t^4+J(\hat{0},\ga_n,\hat{1})t^3\\
&+J(\ga_n,\ga_n,\ga_n)t^3+J(\ga_n,\ga_n,\hat{1})t^2+J(\ga_n,\hat{1},\hat{1})t\\
=&(t^2-(n-1)t+1)(t+1)^2(t-1)^2-(t^5-2t^3+t)\\
=&(t^2-nt+1)(t+1)^2(t-1)^2\\
\end{align*} which is the desired formula.\end{proof}

Now we consider a decomposition of $\M (L,t)$ for a finite lattice $L$. So, if $L$ is a finite lattice then $L^{\op}$ is the same underlying set as $L$ but with the order reversed (i.e. $x\leq^\op y$ in $L^\op$ if and only if $x\geq y$ in $L$). Also for $y\in L$ let $L_y=\{x\in L| x\leq y\}$ and $L^y=\{x\in L| x\geq y\}$. Now we can state the result.

\begin{proposition}\label{M-decomp}

If $L$ is a finite ranked lattice then $$\M(L,t)=t^{\rk (L)}\sum\limits_{y\in L}t^{\crk(y)}\chi(L^y,t)\chi ((L^\op)^y,t^{-1}).$$

\end{proposition}

\begin{proof}

First we note that for $x\leq y\in L$ the M\"obius function on $L^\op$ has $\gm^\op (y,x)=\gm (x,y)$ and that rank is corank in $L^\op$. Then again using Theorem \ref{decomp} we compute \begin{align*} \M (L,t)=&\sum\limits_{(x,y,z)\in\Fl^3(\calP)}J(x,y,z)t^{\gr (x,y,z)}\\
=&\sum\limits_{y\in L}\sum\limits_{x\leq y}\sum\limits_{z\geq y}\gm(x,y)\gm(y,z)t^{\crk(x)+\crk(y)+\crk(z)}\\
=&\sum\limits_{y\in L}t^{\crk(y)}\sum\limits_{x\leq y}\gm(x,y)t^{\crk (x)}\sum\limits_{z\geq y}\gm(y,z)t^{\crk(z)}\\
=&\sum\limits_{y\in L}t^{\crk(y)}\chi(L^y ,t)\sum\limits_{x\leq y}\gm(x,y)t^{\rk (L)-\rk (x)}\\
=&\sum\limits_{y\in L}t^{\crk(y)}\chi(L^y ,t)t^{\rk(L)}\sum\limits_{x\geq^\op y}\gm^\op(y,x)t^{-\rk (x)}\\
=&t^{\rk(L)}\sum\limits_{y\in L}t^{\crk(y)}\chi(L^y ,t)\chi ((L^\op)^y,t^{-1}).
\end{align*} \end{proof}

We can use Proposition \ref{M-decomp} to compute $\M(\calP,t)$ for cases where $\chi (\calP,t)$ is well known. Let $L_n^q$ be the modular lattice of all subspaces in $\F_q^n$, a vector space of dimension $n$ over a field with $q$ elements. The M\"obius function and the characteristic polynomial of $L_q^n$ are well known.

\begin{proposition}[Proposition 7.5.3 \cite{Zas87}]\label{q-basic}

In $L_q^n$ we have  $$\gm(\hat{0},\hat{1})= (-1)^nq^{{n\choose 2}}$$ and $$\chi(L_q^n,t)=\prod\limits_{i=0}^{n-1}(t-q^i).$$

\end{proposition}

Using this we can get a nice formulation for $\M (L_q^n,t)$. First we need to recall so terminology from $q$-series. Let $$\qbin{n}{k} =\frac{(q^n-1)\cdots (q-1)}{(q^k-1)\cdots (q-1)\cdot (q^{n-k}-1)\cdots (q-1)}$$ be the q-binomial coefficient (aka Gaussian coefficient). Also, we denote by $$\qbin{n}{k_1,k_2,\dots,k_m}=\qbin{n}{k_1}\qbin{n-k_1}{k_2}\cdots \qbin{n-(k_1+\cdots k_{m-1})}{k_m}$$ as the q-multinomial coefficient. We also use the q-Pochhammer symbol $$(a;q)_n=\prod\limits_{i=0}^{n-1} (1-aq^i).$$ We use \cite{And86} for a general reference for $q$-series. Using Proposition \ref{q-basic} we get the following.

\begin{proposition}\label{q-M-poly1}

If $L_q^n$ is the modular lattice of subspaces of $\F_q^n$ then $$\M(L_q^n,t)=\sum\limits_{0\leq i\leq j\leq k\leq n}(-1)^{k-i}\qbin{n}{i,j-i,k-j,n-k} q^{{j-i\choose 2}+{k-j\choose 2}}t^{3n -i-j-k}.$$

\end{proposition}

\begin{proof}

Use that $\qbin{n}{k}$ counts the number of subspaces of dimension $k$ in $\F_q^n$ and apply Proposition \ref{decomp} to $J$ in $\M(L_q^n,t)$ together with Proposition \ref{q-basic}. \end{proof}

Now we can reformulate Proposition \ref{q-M-poly1} using Proposition \ref{M-decomp} together with Proposition \ref{q-basic} to get a nice identity in q-series.

\begin{proposition}\label{q-M-decomped}

If $L_q^n$ is the modular lattice of subspaces of $\F_q^n$ then $$\M(L_q^n,t)=t^n\sum\limits_{0\leq  k\leq n}t^{n-k}\qbin{n}{k} \prod_{i=0}^{n-k-1}(t-q^i)\prod_{j=0}^{k-1}(t-q^j).$$

\end{proposition}

It turns out that $-1$ is a root of $\M(L_q^n,t)$. We need a few results in oder to prove this. First we present a formula or $q-$identity which seems to be a kind of $q-$ generalized binomial theorem (the authors could not find it in the literature). It's interesting that in the odd case the sum trivially collapses but not for the even case.  

\begin{lemma}\label{q-john}

If $n>0$ then $$\sum\limits_{k=0}^n (-1)^k \qbin{n}{k} (-1:q)_{n-k}(-1;q)_k=0.$$

\end{lemma}

\begin{proof}

Let $$S(n)=\sum\limits_{k=0}^{n-1}(-1)^k\qbin{n}{k} \frac{(-1;q)_{n-k}(-1;q)_k}{(-1;q)_n}$$ which is the left hand side up to the $n-1$ term divided by the $n^{th}$ term. Using techniques from \cite{PaRi95} and Mathematica  \cite{Mathematica} we build a recursion for $S(n)$. We compute 
\begin{align*} (1+q^{n-1})S(n)=& \sum\limits_{k=0}^{n-1} (-1)^k \left( q^k\qbin{n-1}{k}+\qbin{n-1}{k-1}\right)  \frac{(-1;q)_{n-k}(-1;q)_k}{(-1;q)_{n-1}}\\
=&  \sum\limits_{k=0}^{n-1}(-1)^k\qbin{n-1}{k}  \frac{(-1;q)_{n-k-1}(-1;q)_k}{(-1;q)_{n-1}}(q^k+q^{n-1}) \\
&+\sum\limits_{k=1}^{n-1} (-1)^k\qbin{n-1}{k-1}  \frac{(-1;q)_{n-k}(-1;q)_k}{(-1;q)_{n-1}}\\
=& (-1)^{n-1} 2 q^{n-1} +  \sum\limits_{k=0}^{n-2}(-1)^k\qbin{n-1}{k}  \frac{(-1;q)_{n-k-1}(-1;q)_k}{(-1;q)_{n-1}}q^k\\
&+\sum\limits_{k=0}^{n-2}(-1)^k\qbin{n-1}{k}  \frac{(-1;q)_{n-k-1}(-1;q)_k}{(-1;q)_{n-1}}q^{n-1}\\
&+\sum\limits_{k=0}^{n-2}(-1)^{k+1}\qbin{n-1}{k} \frac{(-1;q)_{n-(k+1)}(-1;q)_{k+1}}{(-1;q)_{n-1}}\\
=& (-1)^{n-1} 2 q^{n-1} +  \sum\limits_{k=0}^{n-2}(-1)^k\qbin{n-1}{k}  \frac{(-1;q)_{n-k-1}(-1;q)_k}{(-1;q)_{n-1}}q^k\\
&+q^{n-1} S(n-1)-\sum\limits_{k=0}^{n-2}(-1)^{k}\qbin{n-1}{k} \frac{(-1;q)_{n-k-1}(-1;q)_{k}}{(-1;q)_{n-1}}(1+q^k)\\
=& (-1)^{n-1} 2 q^{n-1} +q^{n-1} S(n-1) -S(n-1) .\\
\end{align*} Now we prove with induction that $S(n)=(-1)^{n-1}$. First we see that $S(1)=1$. Then using the recursion above we have $$(1+q^{n-1})S(n)=(-1)^{n-1} 2 q^{n-1} -q^{n-1} (-1)^{n-1} +(-1)^{n-1}=(-1)^{n-1}(q^{n-1}+1) $$ which finishes the proof.\end{proof}

\begin{proposition}\label{q-root}

If $L_q^n$ is the modular lattice of subspaces of $\F_q^n$ then $\M(L_q^n,-1)=0$.

\end{proposition}

\begin{proof}

Evaluate the expression in Proposition \ref{q-M-decomped} and apply Lemma \ref{q-john}. \end{proof}


%

%
%
%
%
%

Now we can prove the main result of this section.

\begin{theorem}\label{modto-1}

If $L$ is a modular geometric lattice (modular matroid) then $\M(L,-1)=0$.

\end{theorem}

\begin{proof}

Use the classical result that a modular geometric lattice is product of Boolean and projective spaces (see 12.1 Theorem 4 in \cite{W76} or Proposition 6.9.1 in \cite{Oxley}). Then the result follows from Propositions \ref{q-root}, \ref{M-boolean}, and \ref{prod-J-mob}.\end{proof}

\begin{remark} The proof of Theorem \ref{modto-1} is done in cases. It would be interesting if there was a case free proof just using the modular property. \end{remark}

\begin{remark} At first when looking at examples of $\M$ on matroids it seems that the converse of Theorem \ref{modto-1} might be true. However, the converse is false, but the example seems rather special. Using the SageMath computer algebra system \cite{sage} we compute $$\M(M^*(K_{3,3}),t)=(t^{10} - 9t^9 + 22t^8 + 12t^7 - 81t^6 + 21t^5 + 69t^4 - 18t^3 -
34t^2 + 15t - 1)(t + 1)(t - 1)$$ where $M^*(K_{3,3})$ is the dual matroid of the graphic matroid corresponding to the complete bipartite graph $K_{3,3}$. Since $M^*(K_{3,3})$ is a connected non-modular matroid (it does not have a modular direct summand) this example gives a connected non-modular matroid that has $-1$ as a root of $\M$. This example and Theorem \ref{modto-1} leads to a few questions.\end{remark}

\begin{question} Is there a rank 3 non-modular connected matroid $M$ such that $\M(M,-1)=0 $?\end{question}

\begin{question}  Is there a classification of all matroids who's $\M$ polynomial has -1 as a root?\end{question}

\begin{question} Is there a nice enumerative combinatorial interpretation for $\M(M,-1)$ where $M$ is a matroid (i.e. what does it count)?\end{question}

\subsection{No Deletion-Contraction} We now show that $\J$ and $\M$ are not some evaluation of the Tutte polynomial for matroids. We first recall the following definition.

\begin{definition}\label{del/con}

We say that a function $f$ from matroids to a ring $R$ is a \emph{generalized Tutte-Grothendieck invariant} (following \cite{Ardila-15} Sec 1.8.6) if there exists $a,b\in R$ such that for every matroid $M$ and element of the ground set $e\in M$ $$f(M)=\left\{ \begin{array}{ll}f(M\bs e)f(L) & \text{ if } L \text{ is a loop}\\
 f(M/ e)f(c) & \text{ if } c \text{ is a coloop}\\
 af(M\bs e)+bf(M/e) &\text{ otherwise.}\end{array}\right.$$
 
\end{definition}

Let $U_{r,n}$ be the uniform matroid of rank $r$ on $n$ elements and recall that $U_{r,r}\cong B_r$ are Boolean or free matroids. Then direct computation gives $\J(B_1,t)=t+1$ and $$\J(U_{2,n},t)=(n-1)t^2+nt+n-1.$$ Hence $J(U_{2,3},t)=2t^2+3t+2$. Then any deletion is $U_{2,3}\bs e\cong B_2$ and any contraction is $U_{2,3}/e\cong B_1$. Putting this together with Definition \ref{del/con} and assuming that $\J$ is a Tutte-Grothendieck invariant 
$$2t^2+3t+2=a(t^2+2t+1)+b(t+1).$$ However, this is a contradiction since $t+1$ is not a factor of the right hand side. 

The same result for $\M$ needs two more steps. Looking at the same matroid and using \ref{rank2M} we get $$\M(U_{2,3},t)=(t^2-3t+1)(t+1)^2(t-1)^2=a(t+1)^2(t-1)^4+b(t+1)(t-1)^2$$ which reduces to $$b=(t+1)(t^2-3t+1)-a(t+1)(t-1)^2.$$ Then we look at $U_{2,4}$ and again assume $\M$ is a Tutte-Grothendieck invariant $$\M(U_{2,4},t)=(t^2-4t+1)(t+1)^2(t-1)^2=a(t^2-3t+1)(t+1)^2(t-1)^2+b(t+1)(t-1)^2.$$ Inserting the above value for $b$ and reducing we get $$ t^2-4t+1=a(t^2-3t+1)+(t^2-3t+1)-a(t-1)^2$$ which gives $a=1$ and makes $b=-t(t+1)$. But then $$\M (U_{3,4},t)=(t-1)(t^8 - 3t^7 - t^6 + 12t^5 - 2t^4 - 12t^3 + 3t^2 + 5t - 1)$$ which does not have a factor of $t+1$. This is a contradiction since the right hand side $$\M(U_{3,4}\bs e,t)-t(t+1)\M(U_{3,4}/e,t)=\M(U_{3,3},t)-t(t+1)\M(U_{2,3},t)$$ does have a $t+1$ factor. 

\subsection{Valuations} We study the invariant $\M$ over matroid subdivisions. One could focus on a wider range combinatorial objects like posets but we are motived by applications to matroid theory. First we recall the basis matroid polytope (using \cite{Ardila-Sanchez} as our general reference for this material). A matroid $M$ can be defined via its set of bases $\B(M)$ which are all the independent sets of $M$ who's size is the rank of $M$. Then the matroid polytope of $M$ is $$P(M)=\mathrm{Conv}\{e_B|B\in \B(M)\}$$ where $e_B=e_{i_1}+\cdots +e_{i_r}$ with $B=\{i_1,\dots, i_r\}$. Now we need a few key definitions to state our main result.

\begin{definition}

A \emph{matroid polyhedral subdivision} of a matroid polytope $P(M)$ is a collection of polyhedra $\{P_i\}$ such that $\bigcup P_i=P(M)$, each $P_i$ is a matroid polytope whose vertices are vertices of $P(M)$, and if for $i\neq j$ if $P_i\bigcap P_j\neq \emptyset$ then $P_i\bigcap P_j$ is a proper face of both $P_i$ and $P_j$.

\end{definition}

Now we want to know how invariants decompose across subdivisions which gives rise to valuations. We will use what is called a weak valuation in \cite{Ardila-Sanchez} but we follow \cite{AFR-10} and just say valuation. This makes sense since by Theorem 4.2 in \cite{Ardila-Sanchez} for matroids weak valuations are actually strong valuations.

\begin{definition}

Let $\calP$ be the collection of matroid polytopes and $R$ a commutative ring. A function $f:\calP \to R$ is a (weak) \emph{valuation} if for any matroid polytope $P(M)$ and any matroid polyhedral subdivision with maximal pieces $\{P(M_1),\dots ,P(M_k)\}$ we have that $f(\emptyset)0$ and $$f(P(M))=\sum\limits_{\{j_1,\dots ,j_i\}\subseteq [k]}(-1)^{i}f(P(M_{j_1})\cap \cdots \cap P(M_{j_i})).$$

\end{definition}

Finally we can state the result for the invariant $\J$ in terms of valuations.

\begin{proposition}

The polynomial $\J$ is a valuation on matroids.

\end{proposition}

\begin{proof}

Using Proposition \ref{M-decomp} we know that $$\J(M,t)=(-1)^\rk(M)\sum\limits_{X\in L(M)} \gm(\emptyset ,X)\gm (X,\hat{1})$$ where $\hat{1}$ is the maximal flat of $M$. Hence as a function from the collection of matroids to $\ZZ[t]$ we can represent the function $\J$ as $$\J=(\pm 1)\sum f_1 \star f_2$$ where $f_1 \star f_2=m\circ (f_1\otimes f_2) \circ \Delta_{S>T}$ from the notation in Theorem C in \cite{Ardila-Sanchez} and $f_1=\chi_M (0)$ and $f_2=\chi_M(0)t^{\rk(M)}$. Since $f_1$ and $f_2$ are both Tutte-Grothendieck invariants for matroids and are evaluations of the Tutte polynomial we can conclude that $f_1$ and $f_2$ are both valuations from Proposition 7.5 in \cite{Ardila-Sanchez}. Finally putting it all together Theorem C in \cite{Ardila-Sanchez} finished the result.
\end{proof}

We conclude with a natural question. The polynomial  $\M (L,t)$ is slightly more complicated but has promising properties that seems to imply it should be a valuation.

\begin{question} Is the polynomial $\M$ a matroid valuation? It seems that Proposition \ref{M-decomp} with Proposition 7.5 and Theorem C in \cite{Ardila-Sanchez} is essentially the proof. However that would use that the characteristic polynomial on the flipped lattice of flats $L(M)^{op}$ is a matroid valuation.\end{question}


%
%

\bibliographystyle{amsplain}

\bibliography{whitneyb}

\providecommand{\bysame}{\leavevmode\hbox to3em{\hrulefill}\thinspace}
\providecommand{\MR}{\relax\ifhmode\unskip\space\fi MR }
\providecommand{\MRhref}[2]{%
  \href{http://www.ams.org/mathscinet-getitem?mr=#1}{#2}
}
\providecommand{\href}[2]{#2}
\begin{thebibliography}{10}

\bibitem{AA-17}
Marcelo Aguiar and Federico Ardila, \emph{Hopf monoids and generalized
  permutahedra}, arXiv:1709.075048, 2017.

\bibitem{And86}
George~E. Andrews, \emph{{$q$}-series: their development and application in
  analysis, number theory, combinatorics, physics, and computer algebra}, CBMS
  Regional Conference Series in Mathematics, vol.~66, Published for the
  Conference Board of the Mathematical Sciences, Washington, DC; by the
  American Mathematical Society, Providence, RI, 1986. \MR{858826}

\bibitem{Ardila-15}
Federico Ardila, \emph{Algebraic and geometric methods in enumerative
  combinatorics}, Handbook of enumerative combinatorics, Discrete Math. Appl.
  (Boca Raton), CRC Press, Boca Raton, FL, 2015, pp.~3--172. \MR{3409342}

\bibitem{AFR-10}
Federico Ardila, Alex Fink, and Felipe Rinc\'{o}n, \emph{Valuations for matroid
  polytope subdivisions}, Canad. J. Math. \textbf{62} (2010), no.~6,
  1228--1245. \MR{2760656}

\bibitem{Ardila-Sanchez}
Federico Ardila and Mario Sanchez, \emph{Valuations and the hopf monoid of
  generalized permutahedra}, arXiv:2010.11178, 2020.

\bibitem{Ath-96}
Christos~A. Athanasiadis, \emph{Characteristic polynomials of subspace
  arrangements and finite fields}, Adv. Math. \textbf{122} (1996), no.~2,
  193--233. \MR{1409420}

\bibitem{CF-22}
Amanda Cameron and Alex Fink, \emph{The {T}utte polynomial via lattice point
  counting}, J. Combin. Theory Ser. A \textbf{188} (2022), Paper No. 105584.
  \MR{4369644}

\bibitem{sage}
The~Sage Developers, \emph{{S}age {M}athematics {S}oftware ({V}ersion 8.1)},
  2020, {\tt http://www.sagemath.org}.

\bibitem{Hall-OG}
P.~Hall, \emph{A {C}ontribution to the {T}heory of {G}roups of {P}rime-{P}ower
  {O}rder}, Proc. London Math. Soc. (2) \textbf{36} (1934), 29--95.
  \MR{1575964}

\bibitem{Hardy-Wright}
G.~H. Hardy and E.~M. Wright, \emph{An introduction to the theory of numbers},
  sixth ed., Oxford University Press, Oxford, 2008, Revised by D. R.
  Heath-Brown and J. H. Silverman, With a foreword by Andrew Wiles.
  \MR{2445243}

\bibitem{Mathematica}
Wolfram~Research{,} Inc., \emph{Mathematica, {V}ersion 12.3.1}, Champaign, IL,
  2021.

\bibitem{JR-79}
S.~A. Joni and G.-C. Rota, \emph{Coalgebras and bialgebras in combinatorics},
  Stud. Appl. Math. \textbf{61} (1979), no.~2, 93--139. \MR{544721}

\bibitem{J-12}
Relinde Jurrius, \emph{Relations between {M}\"{o}bius and coboundary
  polynomials}, Math. Comput. Sci. \textbf{6} (2012), no.~2, 109--120.
  \MR{2966347}

\bibitem{KMT-18}
Thomas Krajewski, Iain Moffatt, and Adrian Tanasa, \emph{Hopf algebras and
  {T}utte polynomials}, Adv. in Appl. Math. \textbf{95} (2018), 271--330.
  \MR{3759218}

\bibitem{M-lec-12}
Jeremy~L. Martin, \emph{Lecture notes on algebraic combinatorics}, 2012.

\bibitem{Mobius}
A.~F. M\"{o}bius, \emph{\"{U}ber eine besondere {A}rt von {U}mkehrung der
  {R}eihen}, J. Reine Angew. Math. \textbf{9} (1832), 105--123. \MR{1577896}

\bibitem{M-12}
Will Murray, \emph{M\"{o}bius polynomials}, Math. Mag. \textbf{85} (2012),
  no.~5, 376--383. \MR{3287894}

\bibitem{OT}
Peter Orlik and Hiroaki Terao, \emph{Arrangements of hyperplanes}, Grundlehren
  der Mathematischen Wissenschaften [Fundamental Principles of Mathematical
  Sciences], vol. 300, Springer-Verlag, Berlin, 1992. \MR{1217488}

\bibitem{Oxley}
James Oxley, \emph{Matroid theory}, second ed., Oxford Graduate Texts in
  Mathematics, vol.~21, Oxford University Press, Oxford, 2011. \MR{2849819}

\bibitem{PaRi95}
Peter Paule and Axel Riese, \emph{A {M}athematica {$q$}-analogue of
  {Z}eilberger's algorithm based on an algebraically motivated approach to
  {$q$}-hypergeometric telescoping}, Special functions, {$q$}-series and
  related topics ({T}oronto, {ON}, 1995), Fields Inst. Commun., vol.~14, Amer.
  Math. Soc., Providence, RI, 1997, pp.~179--210. \MR{1448687}

\bibitem{PILZ-BOOK}
G\"{u}nter Pilz, \emph{Near-rings}, second ed., North-Holland Mathematics
  Studies, vol.~23, North-Holland Publishing Co., Amsterdam, 1983, The theory
  and its applications. \MR{721171}

\bibitem{rota64}
Gian-Carlo Rota, \emph{On the foundations of combinatorial theory. {I}.
  {T}heory of {M}\"obius functions}, Z. Wahrscheinlichkeitstheorie und Verw.
  Gebiete \textbf{2} (1964), 340--368 (1964). \MR{0174487}

\bibitem{SO97}
Eugene Spiegel and Christopher~J. O'Donnell, \emph{Incidence algebras},
  Monographs and Textbooks in Pure and Applied Mathematics, vol. 206, Marcel
  Dekker, Inc., New York, 1997. \MR{1445562 (98g:06001)}

\bibitem{Stanley-book1}
Richard~P. Stanley, \emph{Enumerative combinatorics. {V}olume 1}, second ed.,
  Cambridge Studies in Advanced Mathematics, vol.~49, Cambridge University
  Press, Cambridge, 2012. \MR{2868112}

\bibitem{Terao-80}
Hiroaki Terao, \emph{Free arrangements of hyperplanes and unitary reflection
  groups}, Proc. Japan Acad. Ser. A Math. Sci. \textbf{56} (1980), no.~8,
  389--392. \MR{596011}

\bibitem{W-flagincidence}
Max Wakefield, \emph{Partial flag incidence algebras}, preprint,
  arXiv:1605.01685.

\bibitem{Weis}
Louis Weisner, \emph{Abstract theory of inversion of finite series}, Trans.
  Amer. Math. Soc. \textbf{38} (1935), no.~3, 474--484. \MR{1501822}

\bibitem{W76}
D.~J.~A. Welsh, \emph{Matroid theory}, Academic Press [Harcourt Brace
  Jovanovich, Publishers], London-New York, 1976, L. M. S. Monographs, No. 8.
  \MR{0427112}

\bibitem{Z75}
Thomas Zaslavsky, \emph{Facing up to arrangements: face-count formulas for
  partitions of space by hyperplanes}, Mem. Amer. Math. Soc. \textbf{1} (1975),
  no.~issue 1, 154, vii+102. \MR{0357135 (50 \#9603)}

\bibitem{Zas87}
\bysame, \emph{The {M}\"{o}bius function and the characteristic polynomial},
  Combinatorial geometries, Encyclopedia Math. Appl., vol.~29, Cambridge Univ.
  Press, Cambridge, 1987, pp.~114--138. \MR{921071}

\end{thebibliography}

\end{document}